\documentclass[12pt,twoside]{article} 
\usepackage{amsmath,amsxtra,amssymb,latexsym, amscd,amsthm}
\usepackage{graphicx}
\usepackage{subfigure}
\advance\voffset -1.5truecm\relax % top and bottom
\newtheorem{definition}{Definition}[section]
\newtheorem{theorem}{Theorem}[section]
\newtheorem{lemma}{Lemma}[section]

\newtheorem{example}{Example}[section]
\numberwithin{equation}{section}

\begin{document}

\centerline{} 

\centerline{} 

\centerline {\Large{\bf Global behavior of positive solutions }} 

\centerline{} 
\centerline {\Large{\bf of a third order difference equations system}} 
%\centerline {\Large{\bf Qualitative properties for a fourth-order }} 
% \centerline {\Large{\bf Global behavior of positive solutions }} 

 %\centerline{} 
% \centerline {\Large{\bf of a   system    of third-order difference equations }} 

%\centerline {\Large{\bf Qualitative properties for a fourth-order }} 
%\centerline {\Large{\bf Global behavior of positive solutions }} 

%\centerline{} 
 %\centerline {\Large{\bf of a third-order nonlinear system    of difference equations }} 
\centerline{} 
\centerline{ \bf{  Mai Nam Phong  }} 
%,\ Tran Hong Thai$^3$

\centerline{ Department of Mathematical Analysis,} 

\centerline{University of Transport and Communications,} 

\centerline{Hanoi City, Vietnam} 
\centerline{mnphong@utc.edu.vn}

%\centerline{$^3$ Department of Mathematics,} 

%\centerline{Hung Yen University of Techonology  and Education,} 

%\centerline{Hung Yen Province, Vietnam} 

%\centerline{hongthai78@gmail.com}
%\centerline{Address of Author1 forth line} 

%\centerline{} 

%\subjclass{39A10}
%\date{February 14,2004}%\keywords{Equilibrium, asymptotic, positive solution, difference equation, nonoscillatory solution}
\date{}          % Enter your date or \today between curly braces
%\maketitle
\begin{abstract}
In this paper, we consider the  following system of difference equations
\begin{equation*}%\label{E:pt1}
%\begin{cases}
x_{n+1}=\alpha+\dfrac{y_{n}^p}{y_{n-2}^p},\
y_{n+1}=\alpha+ \dfrac{x_{n}^q}{x_{n-2}^q},
%\end{cases}
\ n=0, 1, 2, ...
\end{equation*}
where parameters $\alpha, p, q \in (0, \infty)$ and the initial values $x_{-i}$, $y_{-i}$ are arbitrary positive numbers for $ i=-2,-1, 0$. Our main aim is to investigate  semi-cycle analysis of solutions of above system. Also, we study the boundedness of the positive solutions and the global asymptotic stability of the equilibrium point in case $\alpha>1$, $ 0<p,\ q\leq 1$. Moreover,  the rate of convergence  of the solutions is established. Finally, some numerical examples are given to illustrate our theoretical results.  
\end{abstract}
\noindent
{\bf Keywords:} Semi-cycle, equilibrium, boundedness, global asymptotic stability, rate of convergence.
%\title[Qualitative properties for a ...]{Qualitative properties for a fourth-order rational difference equation (II) }   %(V)   
%\author{Vu Van Khuong$^1$ and Mai Nam Phong$^2$}        % Enter your name between curly braces
%\address{$^1$ Department of Mathematical Analysis\\
%University of transport and communications\\
%Hanoi City, Vietnam}
%\subjclass{39A10}
%\date{February 14,2004}
%\email{vuvankhuong@gmail.com}
%\address{$^2$ Department of Mathematical Analysis\\
%University of transport and communications\\
%Hanoi City, Vietnam}
%\%subjclass{39A10}
%\date{February 14,2004}
%\email{mnphong@gmail.com}

%\keywords{Equilibrium, asymptotic, positive solution, difference equation, nonoscillatory solution}
%\date{}          % Enter your date or \today between curly braces
%\maketitle
\section{Introduction and preliminaries} 
\indent 
%Recently, there has been an increasing  interest in the study of  qualitative analyses  of rational difference equations and systems of difference equations. Difference equations appear naturally as discrete analogues and as numerical solutions of differential and delay differential equations having applications in biology, ecology, economy, physics, and so forth. Although difference equations are very simple in form, it is extremely difficult to understand thoroughly the global behavior of their solutions (see \cite{r1}-\cite{r16} and the references cited therein).\\
Recently, nonlinear difference equations and systems are of wide interest due to their applications in  real life. Such equations appear naturally as the mathematical models which describe biological, physical and economical phenomena.  Although difference equations have very simple forms, however, it is extremely difficult to understand completely the global behavior of their solutions.\\% (see \cite{r1}-\cite{r16} and the references cited therein).\\ 
\indent
In \cite{KaPa1}, Camouzis and Papaschinopoulos  studied the boundedness, persistence and the global behavior of the positive solutions of the following system
\begin{equation*}
x_{n+1}=1+\dfrac{x_n}{y_{n-m}}, y_{n+1}=1+\dfrac{y_n}{x_{n-m}},\ n=0, 1,\ldots
\end{equation*}
where $x_i, y_i$ are positive numbers for $i=-m, -m+1,\ldots, 0$ and $m$ is a positive integer.\\
\indent 
In \cite{Zhanget1}, Zhang et al. investigated the boundedness, persistence and global asymptotic stability of positive solutions of the system of two nonlinear difference equations
\begin{equation*}
 x_{n+1}= A +\dfrac{x_{n-m}}{y_{n }},\ y_{n+1}= B +\dfrac{y_{n-m}}{x_{n}},\ n=0,1,\ldots
 \end{equation*}
where $A, B, x_i, y_i \in (0, \infty)$ for $i=-m, -m+1,\ldots, 0$ and $m \in \mathbb{Z}^+$.\\
\indent 
In \cite{Gumus1}, G\"{u}m\"{u}\c{s} examined the global asymptotic stability of the unique positive equilibrium point and the rate of convergence of positive solutions of the system of two recursive sequences
\begin{equation*}
x_{n+1}= A +\dfrac{y_{n-m}}{y_{n }},\ y_{n+1}= A +\dfrac{x_{n-m}}{x_{n}},\ n=0,1,\ldots,
\end{equation*}
where $A \in (0, \infty)$, $x_i, y_i$ are arbitrary positive numbers for   $i=-m, -m+1,\ldots, 0$ and $m \in \mathbb{Z}^+$.\\
 \indent
 In \cite{Tasdemir}, Ta\c{s}demir  studied the global asymptotic stability of following system of difference equations with quadratic terms   
 \begin{equation*}
 x_{n+1}= A +B\dfrac{y_{n }}{y^2_{n-1 }},\ y_{n+1}= A +B\dfrac{x_{n }}{x^2_{n-1}},\ n=0,1,\ldots,
 \end{equation*}
 where $A$ and $B$ are positive numbers and the initial values $x_i, y_i, $ are positive numbers for $i=-1, 0$.\\
 \indent
 In this paper, motivated by all above mentioned systems we consider the following system of difference equations 
 \begin{equation}\label{E:pt1}
 %\begin{cases}
 x_{n+1}=\alpha+\dfrac{y_{n}^p}{y_{n-2}^p},\
 y_{n+1}=\alpha+ \dfrac{x_{n}^q}{x_{n-2}^q},
 %\end{cases}
 \ n=0, 1, 2, ...
 \end{equation}
 where $\alpha, p, q \in (0, \infty)$ and the initial values $x_{-i}, y_{-i}  \in (0, \infty), i=-2,-1, 0$. More precisely, we investigate semi-cycle analysis of solutions of \eqref{E:pt1}. In addition, we study the boundedness of the positive solutions and the global asymptotic stability of the unique equilibrium point  in case $\alpha>1$, $ 0<p,\ q\leq 1$. Furthermore, we examine the rate of convergence  of the solutions of \eqref{E:pt1}. Finally, some numerical examples are given to verify our theoretical results.\\
 \indent 
We now present some definitions and known theorems which will be used in this paper.\\
\indent
Let $I$ be some interval of real numbers and let
\begin{equation}\label{E:pt7}
f,\ g: I\times I\longrightarrow I
\end{equation}
be continuously differentiable functions. Then, the system of difference equations
\begin{equation}\label{E:pt8}
x_{n+1}=f(y_n,y_{n-2}),\ y_{n+1}=g(x_n,x_{n-2}),\ n=0,1,2,\ldots,
\end{equation}
has a unique solution $\{(x_n,y_n)\}^\infty_{n= 1}$ corresponds to initial values $(x_k, y_k)\in I\times I$ for $k=-2, -1, 0$. 
\begin{definition}(see, \cite{Yalcinkaya})
A point $(\bar{x},\bar{y})$ is called an equilibrium point of the system \eqref{E:pt8} if 
\begin{equation}\label{E:pt9}
\bar{x}=f(\bar{y},\bar{y}),\ \bar{y}=g(\bar{x},\bar{x}).
\end{equation}
\end{definition}
\indent
It is easy to see that the system \eqref{E:pt1} has an unique positive equilibrium point $(\bar{x},\bar{y})=(\alpha+1, \alpha+1)$.
\begin{definition}(see, \cite{Yalcinkaya})
	Let $(\bar{x},\bar{y})$ be a positive equilibrium point of the system \eqref{E:pt8}.\\
	\indent
	A "string" of consecutive terms $\{x_s,\ldots,\ x_m\}$ (resp., $\{y_s,\ldots,\ y_m\}$), $s\geq -2$, $m\leq \infty$ is said to be a positive semi-cycle if $x_i\geq \bar{x}$ (resp., $y_i\geq \bar{y}$),  $i\in \{s,\ldots,m\}$, $x_{s-1}<\bar{x}$ (resp., $y_{s-1}<\bar{y}$), and $x_{m+1}<\bar{x}$ (resp., $y_{m+1}<\bar{y}$).  \\
	\indent
	A "string" of consecutive terms $\{x_s,\ldots,\ x_m\}$ (resp., $\{y_s,\ldots,\ y_m\}$),\ $s\geq -2$, $m\leq \infty$) is said to be a negative semi-cycle if $x_i<\bar{x}$ (resp., $y_i<\bar{y}$),   $i\in \{s,\ldots,m\}$, $x_{s-1}\geq\bar{x}$ (resp., $y_{s-1}\geq\bar{y}$), and $x_{m+1}\geq\bar{x}$ (resp., $y_{m+1}\geq\bar{y}$).  \\
	\indent
	A "string" of consecutive terms $\{(x_s,y_s),\ldots,\ (x_m,y_m)\}$ is said to be a positive semi-cycle (resp., negative semi-cycle)  if both $\{x_s,\ldots,\ x_m\}$ and $\{y_s,\ldots,\ y_m\}$ are positive semi-cycles (resp., negative semi-cycles).  
	\indent
\end{definition}
 %%%%%%%%%%%%%%%%%%%%%%%%%%
 \begin{definition}(see, \cite{Yalcinkaya})
 	Let $(\bar{x},\bar{y})$ be a positive equilibrium point of the system \eqref{E:pt8}.\\
 	\indent
 	A sequence $x_n$ (resp., $y_n$) is said to oscillate    about $\bar x$ (resp., $\bar y$) if for every  $n_0\in \mathbb{N}$ there exist   $l, m \in\mathbb{N}$, $l\geq n_0$, $m\geq n_0$ such that 
 	\begin{equation*}
 	(x_l-\bar x)(x_m-\bar x)\leq 0\ (\text{resp.}, (y_l-\bar y)(y_m-\bar y)\leq 0).  
 	\end{equation*}
 	We say that a   solution $\{(x_n,y_n)\} $ of system \eqref{E:pt8} oscillates   about $(\bar x, \bar y)$ if  $x_n$   oscillates   about $\bar x$  or $y_n$   oscillates   about $\bar y$. 
 \end{definition}
 %%%%%%%%%%%%%%%%%%%%%%%%%%%%%%%%
\begin{definition}(see, \cite{KoLa})
Let $(\bar{x},\bar{y})$ be an equilibrium point of a map $F=(f, g)$, where $f$ and $g$ are continuously differentiable functions at  $(\bar{x},\bar{y})$. The linearized system of  \eqref{E:pt1} about the  equilibrium point  $(\bar x, \bar y)$ is 
	\begin{equation*}
X_{n+1}=F(X_n)=AX_n,   
\end{equation*}
where $	X_n=\left(
x_n, x_{n-1}, x_{n-2}, y_n, y_{n-1}, y_{n-2}\right)^T$ and $A$ is a
  Jacobian matrix of the system  \eqref{E:pt1} about the  equilibrium point  $(\bar x, \bar y)$.  
\end{definition}
\begin{theorem}(see, \cite{KoLa}) \label{T:dl1} 
 For the system $X_{n+1}  =F(X_n),\ n=0, 1,\ldots,$ of difference equations such that $\bar X$ is a fixed point of $F$. If all the eigenvalues of the  Jacobian matrix $A$ about   $ \bar X$ lie inside the open unit disk $|\lambda|<1$, then $ \bar X$  is locally asymptotically stable. If one  of them has a modulus greater than one, then  $ \bar{X} $ is unstable. 
\end{theorem}
\indent
Some related systems of difference equations can be found in \cite{Zhanget2}-\cite{Stevic4}. More results concerning  difference equations and their systems are included in \cite{KoLa, Elaydi1, Elaydi2}. % and the references cited therein.\\
%We now restate some new definitions in \cite{r15}.  
\section{Semi-cycle analysis of \eqref{E:pt1}}
In this section, we study   the behavior of positive solutions of system \eqref{E:pt1} by using  semi-cycle analysis. It is clearly that system \eqref{E:pt1} has a unique possitive equilibrium point $(\bar{x}, \bar{y})=(\alpha+1, \alpha+1)$.%The proof of Lemma 2.1 is clear from \eqref{E:pt1}. So, it will be omitted.\\
\begin{lemma}\label{L:bd1}
Assume that $\{(x_n,y_n)\}_{n=-2}^\infty$ is a solution of the system \eqref{E:pt1}. Then, either $\{(x_n,y_n)\}_{n=-2}^\infty$ is non-oscillatory or it oscillates about the  equilibrium  $(\bar{x}, \bar{y})=(\alpha+1, \alpha+1)$ with semi-cycles that if there exists a semi-cycle with at least two terms, then every consecutive semi-cycle has at least three terms. 
\end{lemma}
\begin{proof}
	Suppose that $\{(x_n,y_n)\}_{n=-2}^\infty$ is a solution of the system \eqref{E:pt1}, and there exists $n_0\geq 0$ such that $(x_{n_0},y_{n_0})$ is the final term of a semi-cycle that has at least two terms. Then, the following two cases can be occurred.\\
	Case $1$: $\ldots, x_{n_0-1}, x_{n_0}<\alpha+1\leq x_{n_0+1}$ and $\ldots, y_{n_0-1}, y_{n_0}<\alpha+1\leq y_{n_0+1}$.\\
	Case $2$: $\ldots, x_{n_0-1}, x_{n_0}\geq \alpha+1 > x_{n_0+1}$ and $\ldots, y_{n_0-1}, y_{n_0}\geq \alpha+1 > y_{n_0+1}$.\\
	We now examine the first case, the second case is similar and will be neglected. Assume that $\ldots, x_{n_0-1}, x_{n_0}<\alpha+1\leq x_{n_0+1}$  and $\ldots, y_{n_0-1}, y_{n_0}<\alpha+1\leq y_{n_0+1}$. Then, from \eqref{E:pt1} we obtain
	 \begin{equation}\label{E:pt22}
	x_{n_0+2}=\alpha+\dfrac{y_{n_0+1}^p}{y_{n_0-1}^p}>\alpha+1,\
	y_{n_0+2}=\alpha+\dfrac{x_{n_0+1}^q}{x_{n_0-1}^q}>\alpha+1.
	 \end{equation}
	 Similarly to \eqref{E:pt22}, we have $x_{n_0+3}>\alpha+1$ and $y_{n_0+3}>\alpha+1$. Hence, the semi-cycle begining with $(x_{n_0+1},y_{n_0+1})$ has at least three terms. Suppose that the semi-cycle which starts with $(x_{n_0+1},y_{n_0+1})$ has length three, then the next semi-cycle will begin with $(x_{n_0+4},y_{n_0+4})$ such that $ x_{n_0+1}, x_{n_0+2}, x_{n_0+3}\geq \alpha+1> x_{n_0+4}$  and $ y_{n_0+1}, y_{n_0+2}, y_{n_0+3}\geq \alpha+1> y_{n_0+4}$, then 
	  \begin{equation}\label{E:pt23}
	 x_{n_0+5}=\alpha+\dfrac{y_{n_0+4}^p}{y_{n_0+2}^p}< \alpha+1,\
	 y_{n_0+2}=\alpha+\dfrac{x_{n_0+4}^q}{x_{n_0+2}^q}<\alpha+1.
	 \end{equation}
	 By arguing similar to \eqref{E:pt23}, we have  $x_{n_0+6}<\alpha+1$ and $y_{n_0+6}<\alpha+1$. From above arguments, we can conclude that the semi-cycle containing $(x_{n_0+1},y_{n_0+1})$ and every semi-cycle after that has at least three terms. The proof is completed.   
\end{proof}
\begin{lemma}\label{L:bd2}
System \eqref{E:pt1} has no nontrivial two periodic solutions.
\end{lemma}
\begin{proof}
	Suppose that system \eqref{E:pt1} has a two periodic solution. Then  $ (x_{n-2}, y_{n-2})=(x_{n}, y_{n})$, for all $n\geq 0$. Therefore, 
	 \begin{equation*}%\label{E:pt24}
	x_{n+1}=\alpha+\dfrac{y_{n }^p}{y_{n-2}^p}= \alpha+1,\
	y_{n+1}=\alpha+\dfrac{x_{n }^q}{x_{n-2}^q}=\alpha+1, \text{for all } n\geq 0.
	\end{equation*}
	Hence, the solution $(x_{n}, y_{n})=(\alpha+1, \alpha+1)$ is the equilibrium solution of system  \eqref{E:pt1}.
	\end{proof}
 \begin{lemma}\label{L:bd3}
 	If the system \eqref{E:pt1} has an increasing solution then it is non-oscillatory positive solution.
 \end{lemma}
\begin{proof}
Assume that $\{(x_n,y_n)\}_{n=-2}^\infty$ is an increasing solution of the system \eqref{E:pt1}. Then, either  $\alpha+1\leq x_1$ and $\alpha+1\leq y_1$ or $x_1<\alpha+1 $ and $y_1<\alpha+1$. The first case is trivial so it will be omitted. We now consider the second case. If  $x_1<\alpha+1 $ and $y_1<\alpha+1$, then we can assert that the semi-cycle containing $(x_1,y_1)$ has at most three terms. Assume by contradiction that the negative semi-cycle begining with $(x_1,y_1)$ involves $(x_4,y_4)$. Then from 
\begin{equation*}%\label{E:pt24}
x_{4}=\alpha+\dfrac{y_{3 }^p}{y_{1}^p}< \alpha+1\ \text{and}\ 
y_{4}=\alpha+\dfrac{x_{3 }^q}{x_{1}^q}<\alpha+1 
\end{equation*}
 imply that $y_{3 }<y_{1}$ and $x_{3 }<x_{1}$ which contradicts the fact that the solution is increasing, so any increasing solution of system \eqref{E:pt1}  is non-oscillatory positive solution.   
\end{proof}
 \begin{lemma}\label{L:bd3}
	System \eqref{E:pt1} has no non-oscillatory negative solutions (has no infinite negative semi-cycle).
\end{lemma}
\begin{proof}
Assume by contradiction that system \eqref{E:pt1} has a non-oscillatory  solution  $\{(x_n,y_n)\}_{n=-2}^\infty$ which has an infinite negative semi-cycle, and assume this semi-cycle begins with  $(x_N,y_N)$, where $N\geq -2$. Then for all $n\geq N$, $(x_n,y_n) < (\alpha+1, \alpha+1)$. From,
  \begin{equation*}%\label{E:pt24}
 x_{n+1}=\alpha+\dfrac{y_{n }^p}{y_{n-2}^p}<\alpha+1,\
 y_{n+1}=\alpha+\dfrac{x_{n }^q}{x_{n-2}^q}<\alpha+1, \text{for all } n\geq N+2,%\max \{1, N-1\},
 \end{equation*}
 we imply $y_n<y_{ n-2}$ and $x_n<x_{ n-2}$. So, we have $\alpha<\ldots<x_{ n+2}<x_n< x_{ n-2}<\alpha+1$ and $\alpha<\ldots<y_{ n+2}<y_n< y_{ n-2}<\alpha+1$ for all $n\geq N+2$, which means that $\{x_n\}$, $\{y_n\}$ have two subsequences $\{x_{2n}\}$, $\{x_{2n+1}\}$ and $\{y_{2n}\}$, $\{y_{2n+1}\}$ that are decreasing and bounded from below. Hence, there exist $a_1$, $a_2$, $b_1$, $b_2$ such that  
\begin{equation*}%\label{E:pt15}
\begin{aligned}
\lim_{n\to\infty} x_{2n}=a_1, \ \lim_{n\to\infty}x_{2n+1}=a_2, \\
\lim_{n\to\infty} y_{2n}=b_1,\ \lim_{n\to\infty} y_{2n+1}=b_2.\\
\end{aligned}
\end{equation*}
Thus, $(a_1, b_1)$,  $(a_2, b_2)$ is a periodic solution of period two of system \eqref{E:pt1}, which contradicts Lemma \ref{L:bd2} unless the solution is the trivial solution. Hence, the solution converges to the equilibrium, which is a contradiction, because the solution is diverging from the equilibrium. This proves that system \eqref{E:pt1} has no non-oscillatory negative solutions. 
\end{proof}
\begin{lemma}\label{L:bd4}
System \eqref{E:pt1} has no decreasing non-oscillatory solutions.
\end{lemma}
\begin{proof}
Assume that system \eqref{E:pt1} has a decreasing non-oscillatory solutions. Then we have one of two following cases will be occurred.\\% for any $n_0\geq 0$.\\
	Case $1$: $\ldots\leq x_{3}\leq  x_{2}\leq  x_{1}\leq \alpha+1 $ and  $\ldots\leq y_{3}\leq  y_{2}\leq  y_{1}\leq \alpha+1 $.\\
	Case $2$: There exists a positive integer $n_0 \geq 3$, such that  $\ldots\leq x_{n_0+2 }\leq  x_{n_0+1} \leq  \alpha+1\leq  x_{n_0 }\leq   x_{n_0-1}\leq \ldots$ and $\ldots\leq y_{n_0+2 }\leq  y_{n_0+1} \leq  \alpha+1\leq  y_{n_0 }\leq   y_{n_0-1}\leq \ldots$. \\
In both cases, the solution has an infinite negative semi-cycle which contradicts Lemma \ref{L:bd3}. Therefore, system  \eqref{E:pt1} has no decreasing non-oscillatory solutions.
\end{proof}
\section{Boundedness and persistence of system \eqref{E:pt1}}
In this section, we will examine the boundedness and persistence of system \eqref{E:pt1} in case $\alpha >1$ and $0<p, q\leq 1$.
 \begin{theorem}\label{T:dl1}
Assume that $\alpha >1$ and $0<p, q\leq 1$. Then every positive solution of system \eqref{E:pt1} is bounded and persists. In detail, we have  
\begin{equation*}
\alpha < x_{2n+2}\leq  x_2a^{n}+ \dfrac{b}{1-a}(1-a^n) 
\end{equation*}
and 
\begin{equation*}
\alpha < x_{2n+3} \leq x_3a^{n}+ \dfrac{b}{1-a}(1-a^n),
\end{equation*}
similarly,
\begin{equation*}
\alpha < y_{2n+2}\leq  y_2a^{n}+ \dfrac{c}{1-a}(1-a^n)
\end{equation*}
and 
\begin{equation*}
\alpha < y_{2n+3}\leq y_3a^{n}+ \dfrac{c}{1-a}(1-a^n),
\end{equation*}
for all $n\geq 0$, where $a=\dfrac{1}{\alpha^{p+q}}$, $b= \alpha^{1-p}+\alpha$ and  $c=\alpha^{1-q}+\alpha$.
\end{theorem}
\begin{proof}
Assume that $\alpha>1$ and  $\{(x_n,y_n)\}_{n=-2}^\infty$ is a positive solution of the system \eqref{E:pt1}. Obviously, from  \eqref{E:pt1} we imply
\begin{equation}\label{E:pt24}
 x_n,y_n >\alpha\   \text{for all}\  n\geq 1.
\end{equation}
By combining \eqref{E:pt1} and \eqref{E:pt24}, we obtain 
\begin{equation}\label{E:pt25}
\begin{aligned}
x_{n }&=\alpha+\dfrac{y_{n-1}^p}{y_{n-3}^p}<\alpha+\dfrac{1}{\alpha^p}y_{n-1}^p<\alpha+\dfrac{1}{\alpha^p}y_{n-1},\\
y_{n }&=\alpha+ \dfrac{x_{n-1}^q}{x_{n-3}^q}<\alpha+\dfrac{1}{\alpha^q}x_{n-1}^q<\alpha+\dfrac{1}{\alpha^q}x_{n-1},
 \end{aligned}
\end{equation}
for all $n\geq 4$.\\
Assume that $\{s_n, t_n\}$ is a solution of the following system
\begin{equation}\label{E:pt26}
%\begin{aligned}
s_{n } = \alpha+\dfrac{1}{\alpha^p}t_{n-1},\ t_{n } = \alpha+\dfrac{1}{\alpha^q}s_{n-1},\ \text{for all}\  n\geq 4,
%\end{aligned}
\end{equation}
 such that 
 \begin{equation}\label{E:pt27}
 %\begin{aligned}
 s_{i } = x_{i},\ t_{i } = y_{i}, i=1, 2, 3.
 %\end{aligned}
 \end{equation}
Now, we use induction to prove that
 \begin{equation}\label{E:pt28}
%\begin{aligned}
x_{n}<s_{n},\ y_{n }<t_{n}, \ \text{for all}\   n\geq 4.
%\end{aligned}
\end{equation}
Clearly, \eqref{E:pt28} is true for $n=4$, suppose that it is true for $n=k>4$. Then, from  \eqref{E:pt25}, we get 
\begin{equation}\label{E:pt29}
\begin{aligned}
x_{k+1 }& <\alpha+\dfrac{1}{\alpha^p}y_{k}<\alpha+\dfrac{1}{\alpha^p}t_{k}=s_{k+1 },\\
y_{k+1 }&<\alpha+\dfrac{1}{\alpha^q}x_{k} <\alpha+\dfrac{1}{\alpha^q}s_{k}=t_{k+1}.
\end{aligned}
\end{equation}
Hence, \eqref{E:pt28} is proven.\\
From \eqref{E:pt26} and \eqref{E:pt27}, we get
  \begin{equation}\label{E:pt29}
  %\begin{aligned}
  s_{n+2} =\dfrac{1}{\alpha^{p+q}}s_{n}+ \alpha^{1-p}+\alpha,\ t_{n+2} = \dfrac{1}{\alpha^{p+q}}t_{n}+ \alpha^{1-q}+\alpha,\ n\geq 2,
  %\end{aligned}
  \end{equation}
for simplicity, let $a=\dfrac{1}{\alpha^{p+q}}$, $b= \alpha^{1-p}+\alpha$ and  $c=\alpha^{1-q}+\alpha$. Then \eqref{E:pt29} turns into
  \begin{equation}\label{E:pt30}
 s_{n+2} =as_{n}+ b,\ t_{n+2} = at_{n}+ c,\ n\geq 2,
 \end{equation}
Solve \eqref{E:pt30}, we obtain
 \begin{equation}\label{E:pt31}
s_{2n+2} =x_2a^{n}+ \dfrac{b}{1-a}(1-a^n),\ s_{2n+3} =x_3a^{n}+ \dfrac{b}{1-a}(1-a^n),\ \text {for all}\  n\geq 0,
\end{equation}
and 
\begin{equation}\label{E:pt32}
t_{2n+2} =y_2a^{n}+ \dfrac{c}{1-a}(1-a^n),\ t_{2n+3} =y_3a^{n}+ \dfrac{c}{1-a}(1-a^n),\ \text {for all}\  n\geq 0.
\end{equation}
Then, from \eqref{E:pt24}, \eqref{E:pt28}, \eqref{E:pt31} and \eqref{E:pt32}, it follows that for all $n\geq 0$, we have 
\begin{equation}\label{E:pt33}
\begin{aligned}
\alpha&< x_{2n+2}\leq  x_2a^{n}+ \dfrac{b}{1-a}(1-a^n),\\
 \alpha&< x_{2n+3} \leq x_3a^{n}+ \dfrac{b}{1-a}(1-a^n),
\end{aligned} 
\end{equation}
and 
\begin{equation}\label{E:pt34}
\begin{aligned}
\alpha&< y_{2n+2} \leq y_2a^{n}+ \dfrac{c}{1-a}(1-a^n),\\
 \alpha&< y_{2n+3} \leq y_3a^{n}+ \dfrac{c}{1-a}(1-a^n).
\end{aligned} 
\end{equation}
The proof is complete.
\end{proof}
\section{Global behavior of system \eqref{E:pt1}}
In the next theorem,  we will show that the unique positive equilibrium is global attractor.
\begin{theorem}\label{T:dl2}
Suppose that $\alpha >1$, $0<p, q\leq 1$. Then every  positive solution of system \eqref{E:pt1} converges to the equilibrium point $(\bar{x},\bar{y})=(\alpha+1,\alpha+1)$ as $n\to\infty$.
\end{theorem}
\begin{proof}
Let  $$l_1=\underset{n\to\infty}{\lim\inf}\ x_n,\  l_2=\underset{n\to\infty}{\lim\inf}\ y_n, $$
$$L_1=\underset{n\to\infty}{\lim\sup}\ x_n,\  L_2=\underset{n\to\infty}{\lim\sup}\ y_n. $$
Obviously, $1<l_1\leq L_1$ and $1<l_2\leq L_2$. From system \eqref{E:pt1} we indicate that 
\begin{equation}\label{E:pt35} 
l_1\geq \alpha+ \dfrac{l_2^p}{L_2^p}\geq \alpha+ \dfrac{l_2 }{L_2 },\ l_2\geq \alpha+ \dfrac{l_1 }{L_1 },\ L_1\leq \alpha+ \dfrac{L_2 }{l_2 },\ L_2\leq \alpha+ \dfrac{L_1 }{l_1 }.
\end{equation}
From \eqref{E:pt35}, we get 
\begin{equation}\label{E:pt36} 
l_1L_2\geq \alpha L_2+ l_2, 
\end{equation}
\begin{equation}\label{E:pt37} 
  l_2L_1\geq \alpha L_1+ l_1, 
\end{equation}
\begin{equation}\label{E:pt38} 
  l_2L_1\leq \alpha l_2+ L_2, 
\end{equation}
\begin{equation}\label{E:pt39} 
 l_1L_2\leq \alpha l_1+ L_1.
\end{equation}
From \eqref{E:pt36} and \eqref{E:pt39} imply that
\begin{equation}\label{E:pt40} 
 \alpha L_2+ l_2\leq \alpha l_1+ L_1.
\end{equation}
From \eqref{E:pt37} and \eqref{E:pt38} indicate that 
\begin{equation}\label{E:pt41} 
\alpha L_1+ l_1\leq \alpha l_2+ L_2.
\end{equation}
From \eqref{E:pt40} and  \eqref{E:pt41}, we obtain
\begin{equation}\label{E:pt42} 
\alpha L_2+ l_2 - \alpha l_2 - L_2\leq \alpha l_1+ L_1-\alpha L_1 - l_1, 
\end{equation}
which is equivalent to 
 \begin{equation}\label{E:pt43} 
 (\alpha-1)(L_2- l_2 + L_1-  l_1)\leq 0.
 \end{equation}
Since $\alpha-1>0$, so we can infer from  \eqref{E:pt43} that $(L_2- l_2 + L_1-  l_1)\leq 0$. But  $L_2- l_2\geq 0$ and $L_1-  l_1\geq 0$, so $(L_2- l_2 + L_1-  l_1)\geq 0$. Hence, $L_2 -l_2 = 0$ and $L_1-  l_1 = 0$, so $L_2 =l_2$ and $L_1=l_1$. Back to  \eqref{E:pt36} and  \eqref{E:pt38}, we get $l_1\geq \alpha +1$ and $L_1\leq \alpha+1$, so $\underset{n\to\infty}{\lim }\ x_n = L_1=l_1= \alpha+1$. Similarly, we imply $\underset{n\to\infty}{\lim }\ y_n = L_2=l_2= \alpha+1$. We complete the proof of the theorem.
\end{proof}
\begin{theorem}\label{T:dl3}
	Assume that $\alpha >1$ and $0<p, q \leq 1$, then the unique positive equilibrium $(\bar x, \bar y)=(\alpha +1, \alpha +1)$ of system \eqref{E:pt1} is locally asymptotically stable.
\end{theorem}
\begin{proof}
Set 
\begin{equation*}
\begin{aligned}
u^{(1)}_n&=x_n, u^{(2)}_n=x_{n-1}, u^{(3)}_n=x_{n-2}, u^{(4)}_n=y_{n }, u^{(5)}_n=y_{n-1}, u^{(6)}_n=y_{n-2},\\
U_n&=\left(u^{(1)}_n, u^{(2)}_n, u^{(3)}_n, u^{(4)}_n, u^{(5)}_n, u^{(6)}_n\right)^T. 
\end{aligned}
\end{equation*}
Then the linearized equation of system \eqref{E:pt1} about the equilibrium point $(\bar x, \bar y)=(\alpha +1, \alpha +1)$ is
$$U_{n+1}=AU_n,$$
where 
\begin{equation*}
\begin{aligned}
U_{n+1}&=\left(u^{(1)}_{n+1}, u^{(2)}_{n+1}, u^{(3)}_{n+1}, u^{(4)}_{n+1}, u^{(5)}_{n+1}, u^{(6)}_{n+1}\right)^T\\
&=\left(\alpha+\dfrac{(u^{(4)}_n)^p}{(u^{(6)}_n)^p}, u^{(1)}_n, u^{(2)}_n, \alpha+\dfrac{(u^{(1)}_n)^q}{(u^{(3)}_n)^q}, u^{(4)}_n, u^{(5)}_n\right)^T, 
\end{aligned}
\end{equation*}
and $A$ is the Jacobian matrix, which is determined by
\begin{equation*}
A  = 
\begin{pmatrix}
0 & 0 & 0 & \dfrac{p}{\alpha+1}&0&-\dfrac{p}{\alpha+1} \\
1 & 0 & 0 & 0&0&0 \\
0 &1 & 0 & 0&0&0 \\
\dfrac{q}{\alpha+1} & 0 & -\dfrac{q}{\alpha+1} & 0&0&0 \\
0 & 0 & 0 & 1&0&0 \\
0 &0 & 0 & 0&1&0 
\end{pmatrix}.
\end{equation*}
Let $\lambda_1, \lambda_2,\ldots,\lambda_6$  denote the eigenvalues of matrix $A$ and let
$$D=diag(d_1, d_2,\ldots, d_6)$$
 be a diagonal matrix in which 
 \begin{equation}\label{E:pt440} 
  d_1=d_4=1,  d_2=d_5=1-2\epsilon, d_3=d_6=1-3\epsilon,
 \end{equation}
  % Since $\alpha >1$ and $0<p, q \leq 1$, we can choose a positive number $\epsilon$ such that 
%\begin{equation}\label{E:pt44} 
% 0<\epsilon<\min\left\{\dfrac{\alpha+1-2p}{3(\alpha+1)},\ \dfrac{\alpha+1-2q}{3(\alpha+1)}\right\}
%\end{equation}
with $0<\epsilon<\dfrac{1}{3}$, so $\det D = (1-2\epsilon)(1-3\epsilon)>0$. Therefore, $D$ is an invertible matrix. Computing matrix $DAD^{-1}$, we have
\begin{equation*}
DAD^{-1} = 
\begin{pmatrix}
0 & 0 & 0 & \dfrac{d_1}{d_4}\dfrac{p}{\alpha+1}&0&-\dfrac{d_1}{d_6}\dfrac{p}{\alpha+1} \\
\dfrac{d_2}{d_1} & 0 & 0 & 0&0&0 \\
0 &\dfrac{d_3}{d_2} & 0 & 0&0&0 \\
\dfrac{d_4}{d_1}\dfrac{q}{\alpha+1} & 0 & -\dfrac{d_4}{d_3}\dfrac{q}{\alpha+1} & 0&0&0 \\
0 & 0 & 0 & \dfrac{d_5}{d_4}&0&0 \\
0 &0 & 0 & 0&\dfrac{d_6}{d_5}&0 
\end{pmatrix}.
\end{equation*}
It is well known that $A$ has the same eigenvalues as $DAD^{-1}$, we have that 
\begin{equation*}
\begin{aligned}
\max|\lambda_i|&\leq \Vert DAD^{-1} \Vert_{\infty}\\  
&=\max\left\{\dfrac{d_2}{d_1},\ \dfrac{d_3}{d_2},\ \dfrac{d_5}{d_4}, \ \dfrac{d_6}{d_5}, \ \dfrac{d_1}{d_4}\dfrac{p}{\alpha+1}+\dfrac{d_1}{d_6}\dfrac{p}{\alpha+1},\  \dfrac{d_4}{d_1}\dfrac{q}{\alpha+1}+\dfrac{d_4}{d_3}\dfrac{q}{\alpha+1}      \right\}.   
\end{aligned}
\end{equation*}
From \eqref{E:pt440}, we imply that 
\begin{equation}\label{E:pt45}
\dfrac{d_2}{d_1}<1,\ \dfrac{d_3}{d_2}<1,\ \dfrac{d_5}{d_4}<1, \ \dfrac{d_6}{d_5}<1.
\end{equation}
Now, we consider%Furthermore, 
\begin{equation}\label{E:pt46}
\begin{aligned}
 \dfrac{d_1}{d_4}\dfrac{p}{\alpha+1}+\dfrac{d_1}{d_6}\dfrac{p}{\alpha+1}&=\dfrac{p}{\alpha+1} +\dfrac{1}{(1-3\epsilon)}\dfrac{p}{(\alpha+1)}\\
 &=\dfrac{p}{\alpha+1}\dfrac{(2-3\epsilon)}{(1-3\epsilon)}<\dfrac{2p}{(\alpha+1)(1-3\epsilon)}<1.\\
 \end{aligned}
 \end{equation}
 Solve \eqref{E:pt46}, we obtain 
 \begin{equation}\label{E:pt47}
 \epsilon<\dfrac{\alpha+1-2p}{3(\alpha+1)}
 \end{equation}
 Similarly, if $\epsilon$ satisfies
 \begin{equation}\label{E:pt48}
 \epsilon<\dfrac{\alpha+1-2q}{3(\alpha+1)}
 \end{equation}
 then 
 \begin{equation}\label{E:pt49}
 \dfrac{d_4}{d_1}\dfrac{q}{\alpha+1}+\dfrac{d_4}{d_3}\dfrac{q}{\alpha+1}<1.
 \end{equation}
 From \eqref{E:pt45}, \eqref{E:pt46}, \eqref{E:pt47}, \eqref{E:pt48} and \eqref{E:pt49}, we can choose $\epsilon$ such that 
 \begin{equation}\label{E:pt50} 
  0<\epsilon<\min\left\{\dfrac{\alpha+1-2p}{3(\alpha+1)},\ \dfrac{\alpha+1-2q}{3(\alpha+1)}\right\},
 \end{equation}
 then
 \begin{equation*}
 \begin{aligned}
 \max|\lambda_i| \leq \Vert DAD^{-1} \Vert_{\infty}<1.
 \end{aligned}
 \end{equation*}
 It means that all eigenvalues of $A$ lie inside the unit disk. This implies that the unique positive equilibrium $(\bar x, \bar y)=(\alpha +1, \alpha +1)$ of system \eqref{E:pt1} is locally asymptotically stable.
 Thus, the proof is completed.
\end{proof}
Combining Theorem  \ref{T:dl2} and   Theorem  \ref{T:dl3}, we get the next theorem.
\begin{theorem}\label{T:dl4}
	If $\alpha >1$ and  $0<p, q\leq 1$, then the unique positive equilibrium point $(\bar{x},\bar{y})=(\alpha+1,\alpha+1)$ of system \eqref{E:pt1} is globally asymptotically stable.
\end{theorem}
 \section{Rate of convergence}
 In this section we give the rate of convergence of a solution that converges to the equilibrium point  $ (\bar{x},\ \bar{y})=(\alpha+1, \alpha+1)$ of the systems \eqref{E:pt1} for $\alpha >1$ and  $0<p, q\leq 1$. More results on the rate of convergence of solutions for some two-dimensional difference equations systems have been obtained in \cite{Kunu1} and \cite{Kunu2}.\\
 \indent
 The following results give the rate of convergence of solutions of a system of difference equations
 \begin{equation}\label{E:pt3.1}
 \mathbf{V}_{n+1}=[A+B(n)]\mathbf{V}_n
 \end{equation}
 where $\mathbf{V}_n$ is a $k$-dimensional vector, $A\in \mathbb{C}^{k\times k}$ is a constant matrix, and $B:\ \mathbb{Z}^+\longrightarrow \mathbb{C}^{k\times k}$ is a matrix function satisfying 
 \begin{equation}\label{E:pt3.2}
 \Vert B(n)\Vert \to 0\ \text{when}\ n\to\ \infty,
 \end{equation}
 where $\Vert . \Vert$ denotes any matrix norm which is associated with the vector norm;  $\Vert . \Vert$ also denotes the Euclidean norm in $\mathbb{R}^2$ given by 
 \begin{equation}\label{E:pt3.3}
  \Vert (x, y)\Vert =\sqrt{x^2+y^2}.
 \end{equation}
 \begin{theorem}(Perron's Theorem, \cite{r15})\label{T:dl3.1}
 	Assume that condition \eqref{E:pt3.2} holds. If $\mathbf{V}_n$ is a solution of system \eqref{E:pt3.1}, then either  $\mathbf{V}_n=0$ for all large $n$, or 
 	\begin{equation*}%\label{E:pt3.4}
 \rho=\underset{n\to\infty}\lim\sqrt[n]{\Vert \mathbf{V}_n\Vert}
 	\end{equation*}
 	or
 	\begin{equation*}%\label{E:pt3.5}
 	 \rho=\underset{n\to\infty}\lim\dfrac{\Vert \mathbf{V}_{n+1}\Vert}{\Vert \mathbf{V}_n\Vert}
 	\end{equation*}
 	exists and $\rho$ is equal to the modulus of one of the eigenvalues of matrix $A$.
 \end{theorem}
 \begin{theorem}
 	Assume that $\{(x_n,y_n)\}$ is a solution of the system \eqref{E:pt1} such that $\underset{n\to\infty}\lim x_n=\bar{x}, \underset{n\to\infty}\lim y_n =\bar{y}$. Then the error vector
 	\begin{equation*}
 	\mathbf{e}_n=
 	\begin{pmatrix}
 	e^1_n\\ 
 	e^1_{n-1}\\
 	e^1_{n-2}\\
 	e^2_n\\ 
 	e^2_{n-1}\\ 
 	e^2_{n-2}\\
 	\end{pmatrix}
 	=
 	\begin{pmatrix}
 	x_n-\bar{x}\\ 
 	x_{n-1}-\bar{x}\\ 
 	x_{n-2}-\bar{x}\\
 	y_n-\bar{y}\\
 	y_{n-1}-\bar{y}\\ 
 	y_{n-2}-\bar{y}\\
 	\end{pmatrix}
 	\end{equation*}
 	of every solution of \eqref{E:pt1} satisfies both of the following asymptotic relations:%$(x_n, y_n) \ne (\bar x, \bar y)$ 
 	\begin{equation*}%\label{E:pt3.10}
 	\underset{n\to\infty}\lim\sqrt[n]{\Vert \mathbf{e}_n\Vert}=|\lambda_i J_F(\bar x, \bar y)|\ \text{for some} \ i\in\{1,\ 2,\ldots, 6\} 
 	\end{equation*}
 	and 
 	\begin{equation*}%\label{E:pt3.11}
 	\underset{n\to\infty}\lim\dfrac{\Vert \mathbf{e}_{n+1}\Vert}{\Vert \mathbf{e}_{n}\Vert}=|\lambda_i J_F(\bar x, \bar y)|\ \text{for some} \ i\in\{1,\ 2,\ldots, 6\}
 	\end{equation*}
 	where  $|\lambda_i J_F(\bar x, \bar y)|$ is equal to the modulus of one  the eigenvalues of the Jacobian matrix evaluated at the equilibrium $(\bar x, \bar y)$.
 \end{theorem}
 \begin{proof}
 	We will find a system satisfied by the error terms, which are given as
 	\begin{equation}\label{E:pt3.12}
 	\begin{aligned}
 	x_{n+1}-\bar{x}=&\sum_{i=0}^{2}a_i(x_{n-i}- \bar x)+ \sum_{i=0}^{2}b_i(y_{n-i}- \bar x)\\ 
 		y_{n+1}-\bar{y}=&\sum_{i=0}^{2}c_i(x_{n-i}- \bar x)+ \sum_{i=0}^{2}d_i(y_{n-i}- \bar x)\\ 
 	\end{aligned}
 	\end{equation}
  	Let $ e^1_n=x_n-\bar{x}\ \text{and}\ e^2_n=y_n-\bar{y}$, then system \eqref{E:pt3.12} can be written as following form\\
 	\begin{equation*}
 	\begin{aligned}
 	e^1_{n+1}= &\sum_{i=0}^{2}a_ie^1_{n-i}+ \sum_{i=0}^{2}b_ie^2_{n-i},\\ 
 		e^2_{n+1}= &\sum_{i=0}^{2}c_ie^1_{n-i}+ \sum_{i=0}^{2}d_ie^2_{n-i},\\
 	\end{aligned}
 	\end{equation*}
 	where 
 	\begin{equation*}
 	\begin{aligned}
 	a_0&=a_1=a_2=0, \\
 	b_0& =\dfrac{py_n^{p-1}}{y_{n-2}^p}, b_1=0, b_2= -\dfrac{py_n^{p }}{y_{n-2}^{p+1}},\\
 	c_0&=\dfrac{qx_n^{q-1}}{x_{n-2}^q}, c_1=0, c_2= -\dfrac{qx_n^{q }}{x_{n-2}^{q+1}},  \\
 		d_0&=d_1=d_2=0.
 	\end{aligned}
 	\end{equation*}
 	Taking the limmits of $a_i,\ b_i,\ c_i$ and $ d_i$ as $n\to\infty$ for $i=0, 1, 2$, we obtain
 	\begin{equation*}
 	\begin{aligned}
 	\underset{n\to\infty}\lim a_i&=0,\ \text{for}\ i=0,1,2,\\ \underset{n\to\infty}\lim b_0&=\dfrac{p}{\bar y}, \underset{n\to\infty}\lim b_1=0, \underset{n\to\infty}\lim b_2=-\dfrac{p}{\bar y},  \\
 	\underset{n\to\infty}\lim c_0&=\dfrac{q}{\bar x}, \underset{n\to\infty}\lim c_1=0, \underset{n\to\infty}\lim c_2=-\dfrac{q}{\bar x},  \\
 	\underset{n\to\infty}\lim d_i&=0,\ \text{for}\ i=0,1,2.\\
 	\end{aligned}
 	\end{equation*} 
 	That is
 	\begin{equation*}
 	\begin{aligned}
 	b_0&= \dfrac{p}{\bar y}+\beta_n,\ b_2 =-\dfrac{p}{\bar y}+\gamma_n,\\
 c_0&=\dfrac{q}{\bar x}+\delta_n,\ c_2=-\dfrac{q}{\bar x}+\eta_n,\\  
 	\end{aligned}
 	\end{equation*} 
 	where $ \beta_n\to 0,\ \gamma_n\to 0, \delta_n\to 0$ and $\eta_n \to 0$   as $n\to\infty$.\\
 	\indent
 	Now, we have the following system of the form \eqref{E:pt3.1}:\\
 	$$ \mathbf{e}_{n+1}=[A+B(n)]\mathbf{e}_n, $$
 	where $	\mathbf{e}_n=\left(
 	e^1_n, e^1_{n-1}, e^1_{n-2}, e^2_n, e^2_{n-1}, e^2_{n-2}\right)^T$ and 
 	%\begin{equation*}
 	$$A=J_F(\bar x, \bar y)= 
 	\begin{pmatrix}
 	0 & 0 & 0 & \dfrac{p}{\alpha+1}&0&-\dfrac{p}{\alpha+1} \\
 	1 & 0 & 0 & 0&0&0 \\
 	0 &1 & 0 & 0&0&0 \\
 	\dfrac{q}{\alpha+1} & 0 & -\dfrac{q}{\alpha+1} & 0&0&0 \\
 	0 & 0 & 0 & 1&0&0 \\
 	0 &0 & 0 & 0&1&0 
 	\end{pmatrix},$$
 	%\end{equation*}
 	$$B(n)=	\begin{pmatrix}
 	0 & 0 & 0 & \beta_n&0&\gamma_n \\
 	0 & 0 & 0 & 0&0&0 \\
 	0 &0 & 0 & 0&0&0 \\
 	\delta_n & 0 & \eta_n& 0&0&0 \\
 	0 & 0 & 0 & 0&0&0 \\
 	0 &0 & 0 & 0&0&0 
 	\end{pmatrix},$$
 	and 
 	$$ \Vert B(n)\Vert\to 0\ \text{as}\ n\to\ \infty. $$
 	Thus, the limiting system of error terms can be written as:
 	$$ \begin{pmatrix}
 	e^1_{n+1}\\ 
 	e^1_{n }\\
 	e^1_{n-1}\\
 	e^2_{n+1} \\ 
 	e^2_{n }\\ 
 	e^2_{n-1}\\
 	\end{pmatrix}=
 	\begin{pmatrix}
 	0 & 0 & 0 & \dfrac{p}{\alpha+1}&0&-\dfrac{p}{\alpha+1} \\
 	1 & 0 & 0 & 0&0&0 \\
 	0 &1 & 0 & 0&0&0 \\
 	\dfrac{q}{\alpha+1} & 0 & -\dfrac{q}{\alpha+1} & 0&0&0 \\
 	0 & 0 & 0 & 1&0&0 \\
 	0 &0 & 0 & 0&1&0 
 	\end{pmatrix}
 	\begin{pmatrix}
 	e^1_n\\ 
 	e^1_{n-1}\\
 	e^1_{n-2}\\
 	e^2_n\\ 
 	e^2_{n-1}\\ 
 	e^2_{n-2}\\
 	\end{pmatrix}
 	.$$
 	The system is exactly linearized system of \eqref{E:pt1} evaluated at the equilibrium $ (\bar{x},\ \bar{y})=(\alpha+1,\ \alpha+1)$. From  Theorem \ref{T:dl3.1}, we imply the result.
 \end{proof}
 
 \section{Examples}
 In order to verify our theoretical results and to support our theoretical discussion, we consider several interesting numerical examples. These examples represent different types of qualitative behavior of solutions of the systems \eqref{E:pt1}. All plots in this section are drawn with Matlab.
 \begin{example}\label{vd:1}
 	Let $\alpha = 2, p = 0.6, q = 0.9$.  The system \eqref{E:pt1} can be written as
 	\begin{equation}\label{E:pt211}
 	x_{n+1}=2+\dfrac{y_{n}^{0.6}}{y_{n-2}^{0.6}},\
 	y_{n+1}=2+ \dfrac{x_{n}^{0.9}}{x_{n-2}^{0.9}}.
 	\end{equation}
 Consider \eqref{E:pt211} with initial conditions $x_{-2} = 2.5$,  $x_{-1} = 6$, $x_0=2$, $y_{-2} = 4$, $y_{-1} = 2$ and $y_0 = 5$.
 \end{example}
 \begin{figure}[!ht]
 	\centering
 	\subfigure[Plot of $x_n$ for the system \eqref{E:pt211}]{
 		\includegraphics[height=4.8cm,width=6.4cm]{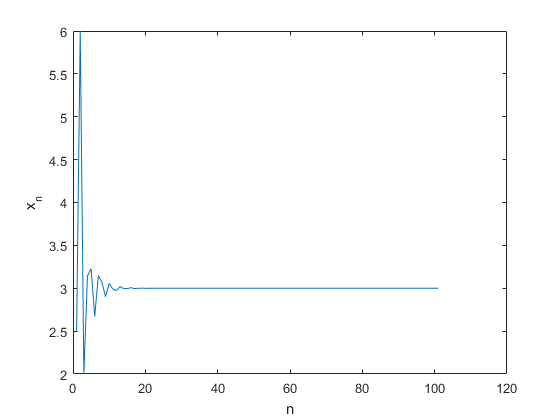}
 		\label{fig:1a}}
 	\hspace{0.2cm}
 	\subfigure[Plot of $y_n$ for the system \eqref{E:pt211}]{
 		\includegraphics[height=4.8cm,width=6.4cm]{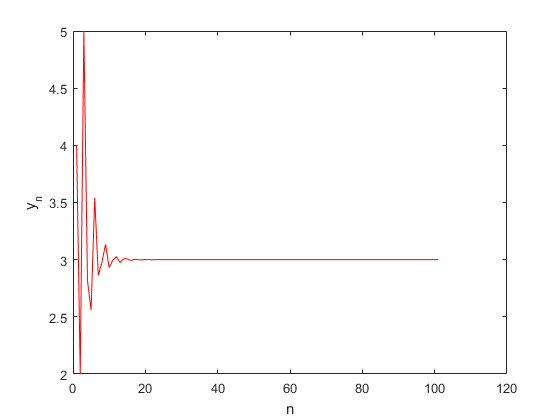}
 		\label{fig:1b}}
 	\hspace{0.2cm}
 	\subfigure[An attractor of the system \eqref{E:pt211}]{
 		\includegraphics[height=4.8cm,width=6.4cm]{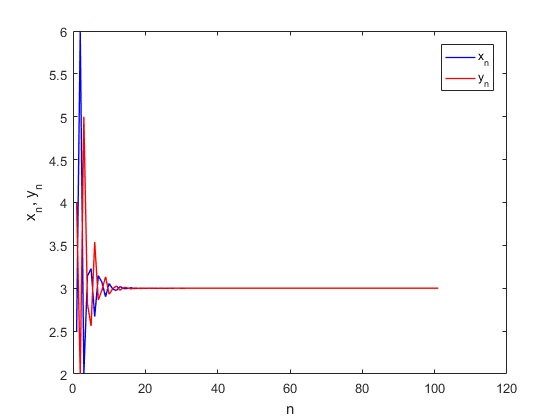}
 		\label{fig:1c}}
 	\caption{Plots for the system \eqref{E:pt211}}\label{fig:1}
 \end{figure}
 In this case, $\alpha>1$ and $0<p,\ q<1$, so the  unique positive equilibrium point $ (\bar{x},\ \bar{y})=(3, 3)$ of the system \eqref{E:pt1}  is globally asymptotically  stable (see Theorem \ref{T:dl4}). In Figure \ref{fig:1}, the plot of $x_n$ is shown in Figure \ref{fig:1} (a), the plot of $y_n$ is shown in Figure \ref{fig:1} (b), and an attractor of the system \eqref{E:pt211} is shown in Figure \ref{fig:1} (c).
 %%%%%%%%%%%%%%%a phase portrait 
  
\begin{example}\label{vd:2}
	Let $\alpha = 1.3, p = 0.9, q = 0.8$.  The system \eqref{E:pt1} can be written as
	\begin{equation}\label{E:pt212}
	x_{n+1}=1.3 + \dfrac{y_{n}^{0.9}}{y_{n-2}^{0.9}},\
	y_{n+1}=1.3 + \dfrac{x_{n}^{0.8}}{x_{n-2}^{0.8}}.
	\end{equation}
Consider \eqref{E:pt212}	with initial conditions $x_{-2} = 2.6$,  $x_{-1} = 1.8$, $x_0=3$, $y_{-2} = 3$, $y_{-1} = 5$ and $y_0 = 1$.
\end{example}
\begin{figure}[!ht]
	\centering
	\subfigure[Plot of $x_n$ for the system \eqref{E:pt212}]{
		\includegraphics[height=4.8cm,width=6.4cm]{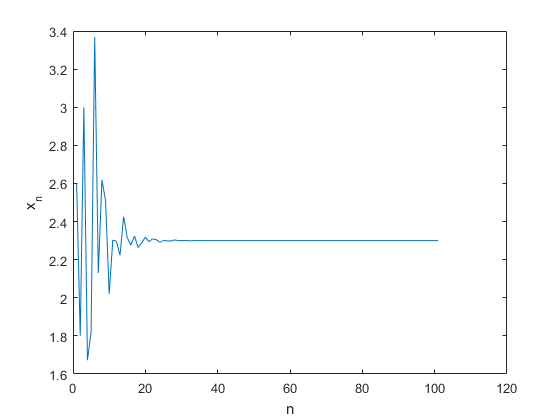}
		\label{fig:2a}}
	\hspace{0.2cm}
	\subfigure[Plot of $y_n$ for the system \eqref{E:pt212}]{
		\includegraphics[height=4.8cm,width=6.4cm]{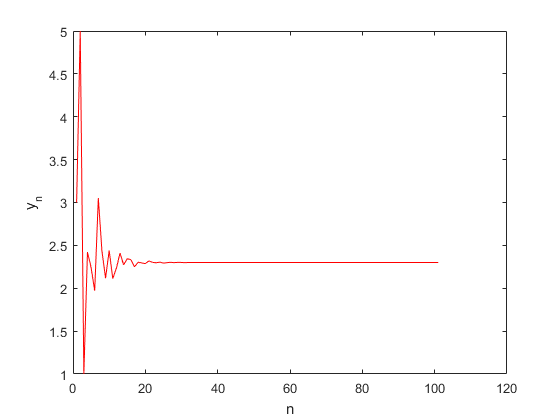}
		\label{fig:2b}}
	\hspace{0.2cm}
	\subfigure[An attractor of the system \eqref{E:pt212}]{
		\includegraphics[height=4.8cm,width=6.4cm]{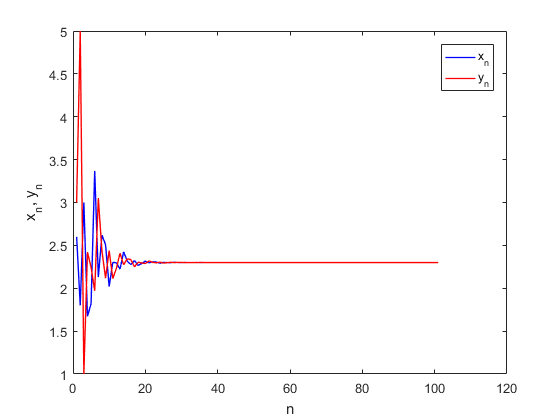}
		\label{fig:2c}}
	\caption{Plots for the system \eqref{E:pt212}}\label{fig:2}
\end{figure}
In this example, the  unique positive equilibrium point of the system \eqref{E:pt1}  is also globally asymptotically stable because $\alpha>1$ and $0<p,\ q<1$  satisfy conditions of  Theorem \ref{T:dl4}. In Figure \ref{fig:2}, the plot of $x_n$ is shown in Figure \ref{fig:2} (a), the plot of $y_n$ is shown in Figure \ref{fig:2} (b), and an attractor of the system \eqref{E:pt212} is shown in Figure \ref{fig:2} (c).
%%%%%%%%%%%%%%%%%%%%%%%%%%%%%%%%%%%%%%%%%%%%%a phase picture
\begin{example}\label{vd:3}
	Let $\alpha = 0.6, p = 0.8, q = 1.9$.  The system \eqref{E:pt1} can be written as
	\begin{equation}\label{E:pt213}
	x_{n+1}=0.6+\dfrac{y_{n}^{0.8}}{y_{n-2}^{0.8}},\
	y_{n+1}=0.6+ \dfrac{x_{n}^{1.9}}{x_{n-2}^{1.9}}. 
	\end{equation}
Examine system \eqref{E:pt213}	with initial conditions $x_{-2} = 1.6$,  $x_{-1} = 2.8$, $x_0=4$, $y_{-2} = 4$, $y_{-1} = 1.5$ and $y_0 = 6$.
\end{example}
\begin{figure}[!ht]
	\centering
	\subfigure[Plot of $x_n$ for the system \eqref{E:pt213}]{
		\includegraphics[height=4.8cm,width=6.4cm]{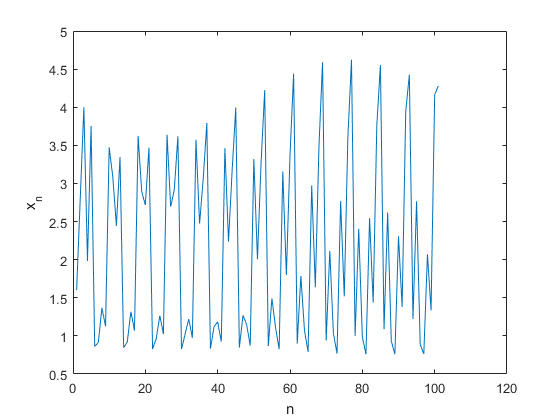}
		\label{fig:3a}}
	\hspace{0.2cm}
	\subfigure[Plot of $y_n$ for the system \eqref{E:pt213}]{
		\includegraphics[height=4.8cm,width=6.4cm]{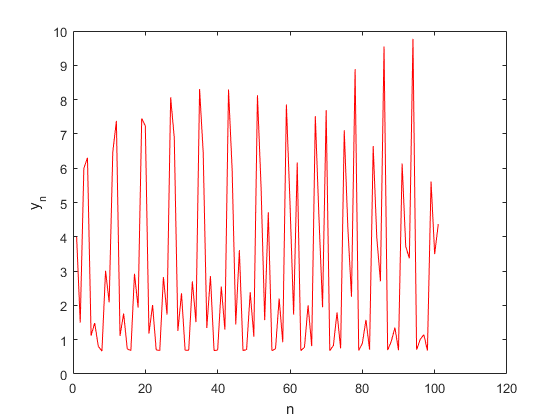}
		\label{fig:3b}}
	\hspace{0.2cm}
	\subfigure[Phase portrait  of the system \eqref{E:pt213}]{
		\includegraphics[height=4.8cm,width=6.4cm]{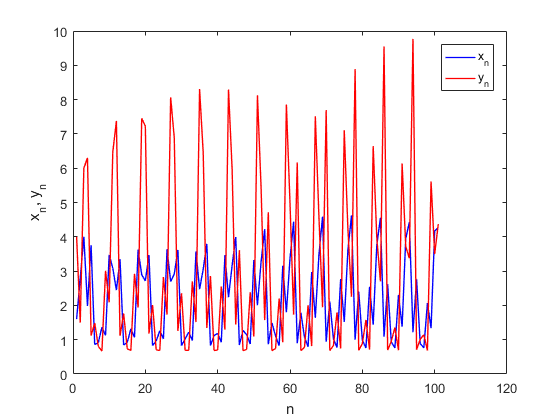}
		\label{fig:3c}}
	\caption{Plots for the system \eqref{E:pt213}}\label{fig:3}
\end{figure}
In this situation, since $\alpha<1$ and $q>1$, so they are not satisfied conditions of  Theorem \ref{T:dl4}. Therefore, the unique positive equilibrium point $ (\bar{x},\ \bar{y})=(1.6, 1.6)$ of the system \eqref{E:pt1}  is not globally asymptotically stable. In Figure \ref{fig:3}, the plot of $x_n$ is shown in Figure \ref{fig:3} (a), the plot of $y_n$ is shown in Figure \ref{fig:3} (b), and a phase portrait of the system \eqref{E:pt213} is shown in Figure \ref{fig:3} (c).
%%%%%%%%%%%%%%%%%%%%%%%%%%%%%%%%%%%%%%%%%%%%%
\begin{example}\label{vd:4}
	Let $\alpha = 0.3, p =1.2, q = 1.5$.  The system \eqref{E:pt1} can be written as
	\begin{equation}\label{E:pt214}
	x_{n+1}=0.3 + \dfrac{y_{n}^{1.2}}{y_{n-2}^{1.2}},\
	y_{n+1}=0.3 + \dfrac{x_{n}^{1.5}}{x_{n-2}^{1.5}}.
	\end{equation}
	Study system \eqref{E:pt214} with initial conditions $x_{-2} =  6$,  $x_{-1} =  8$, $x_0=3$, $y_{-2} = 3$, $y_{-1} =  5$ and $y_0 = 1$.
\end{example}
\begin{figure}[!ht]
	\centering
	\subfigure[Plot of $x_n$ for the system \eqref{E:pt214}]{
		\includegraphics[height=4.8cm,width=6.4cm]{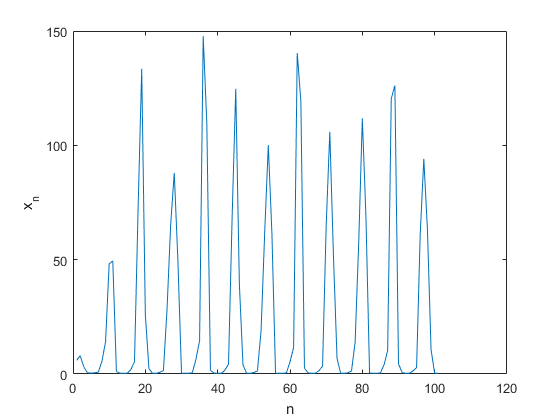}
		\label{fig:4a}}
	\hspace{0.2cm}
	\subfigure[Plot of $y_n$ for the system \eqref{E:pt214}]{
		\includegraphics[height=4.8cm,width=6.4cm]{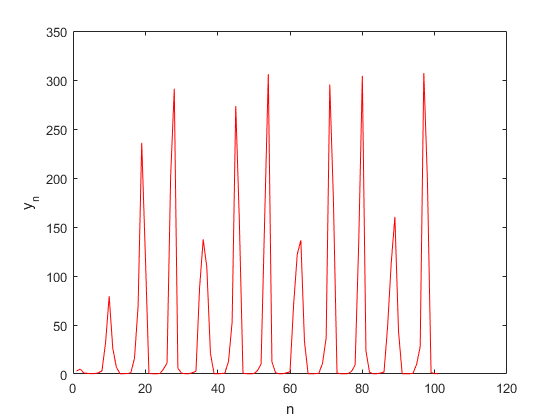}
		\label{fig:4b}}
	\hspace{0.2cm}
	\subfigure[Phase portrait  of the system \eqref{E:pt214}]{
		\includegraphics[height=4.8cm,width=6.4cm]{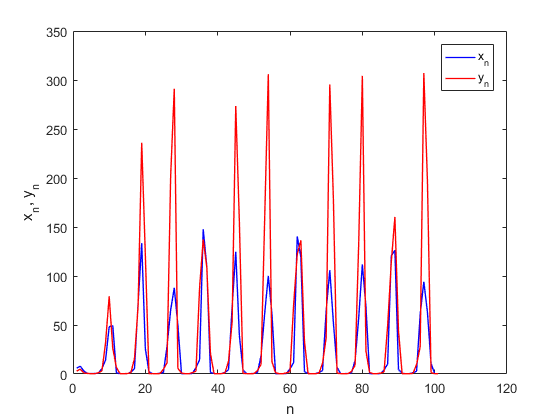}
		\label{fig:4c}}
	\caption{Plots for the system \eqref{E:pt214}}\label{fig:4}
\end{figure}
In this case, because $\alpha<1$, $p, q>1$, so they are not satisfied conditions of  Theorem \ref{T:dl4}, thus the  unique positive equilibrium point $ (\bar{x},\ \bar{y})=(1.3, 1.3)$ of the system \eqref{E:pt1}  is not globally asymptotically stable. In Figure \ref{fig:4}, the plot of $x_n$ is shown in Figure \ref{fig:4} (a), the plot of $y_n$ is shown in Figure \ref{fig:4} (b), and a phase portrait of the system \eqref{E:pt214} is shown in Figure \ref{fig:4} (c).
%x_{n+1}=\alpha+\dfrac{y_{n}^p}{y_{n-2}^p},\
%y_{n+1}=\alpha+ \dfrac{x_{n}^q}{x_{n-2}^q}

\end{document}